\documentclass{amsart}

\usepackage{amsmath,
amsfonts,
amsthm,
amssymb,
bbold,
mathrsfs}

\renewcommand\subset{\subseteq}

\newcommand{\dom}{\operatorname{dom}}
\newcommand{\lex}{\operatorname{lex}}
\newcommand{\comment}[1]{}

\newcommand{\Ht}{\operatorname{ht}}
\newcommand{\PFA}{\mathrm{PFA}}
\newcommand{\Ecal}{\mathcal{E}}
\newcommand{\MA}{\mathrm{MA}}
\newcommand{\Seq}[1]{\langle #1 \rangle}
\newcommand{\FA}{\mathrm{FA}}

\newcommand{\Mscr}{\mathscr{M}}

\newcommand{\Qbb}{\mathbb{Q}}
\newcommand{\Rbb}{\mathbb{R}}

\newcommand{\range}{\mathrm{range}}

\theoremstyle{plain}
\newtheorem{thm}{Theorem}[section]
\newtheorem{lem}[thm]{Lemma}
\newtheorem{prop}[thm]{Proposition}

\newtheorem{fact}[thm]{Fact}
\newtheorem{claim}[thm]{Claim}
\newtheorem{question}[thm]{Question}

\theoremstyle{definition}
\newtheorem{defn}[thm]{Definition}
\newtheorem{rmk}[thm]{Remark}

\newtheorem{notation}[thm]{Notation}

\begin{document}

\author[H. Lamei Ramandi]{Hossein Lamei Ramandi}
\author[J. Tatch Moore]{Justin Tatch Moore}

\address{Department of Mathematics \\ Cornell University\\
Ithaca, NY 14853--4201 \\ USA}

\title[Non $\sigma$-scattered orders]{There may be no minimal non $\sigma$-scattered linear orders}

\subjclass[2010]{Primary: 03E35; Secondary: 03E05, 03E57}
\keywords{linear order, proper forcing, $\sigma$-scattered}

\email{{\tt hossein@math.cornell.edu}}
\email{{\tt justin@math.cornell.edu}}

\begin{abstract}
In this paper we demonstrate that it is consistent, relative to the existence of a supercompact
cardinal, that there is no linear order
which is minimal with respect to being non $\sigma$-scattered.
This shows that a theorem of Laver, which asserts that the class of $\sigma$-scattered linear
orders is well quasi-ordered, is sharp.
We also prove that $\PFA^+$ implies that every non $\sigma$-scattered linear
order either contains a real type, an Aronszajn type, or a ladder system
indexed by a stationary subset of $\omega_1$,
equipped with either the lexicographic or reverse lexicographic order.
Our work immediately implies that $CH$ is consistent with ``no Aronszajn tree has a base of cardinality $\aleph_1$.''
This gives an affirmative answer to a problem due to Baumgartner.
\end{abstract}

\maketitle

\section{Introduction}

In \cite{Fraisse_conj}, Laver verified a longstanding conjecture of Fra\"iss\'e:
the countable linear orders are \emph{well quasi-ordered} by embeddability.
That is to say if $L_i$ $(i < \infty)$ is an infinite sequence of countable linear orderings, then
there is an $i < j$ such that $L_i$ is embeddable into $L_j$.
In fact, Laver proved the following stronger result.

\begin{thm} \label{Laver_thm} \cite{Fraisse_conj}
The class $\mathscr{M}$ of $\sigma$-scattered linear orders is well quasi-ordered by embeddability.
\end{thm}

\noindent
Recall that a linear order is \emph{scattered} if it does not contain an
isomorphic copy of the linear order $(\Qbb,\leq)$ and is \emph{$\sigma$-scattered}
if it is a union of countably many scattered suborders.

In the final paragraph of \cite{Fraisse_conj}, Laver writes,
``Finally, the question arises as to how the order types outside of
$\Mscr$ behave
under embeddability.''
For instance, is there a class of linear orders which is closed under taking suborders,
which properly includes the class of $\sigma$-scattered linear orders,
and which is well quasi-ordered by embeddability?
Cast in another way, is there a non $\sigma$-scattered linear order which embeds into 
all of its non $\sigma$-scattered suborders?

Already in \cite{reals_isomorphic}, Baumgartner proved that it is consistent that
any two \emph{$\aleph_1$-dense} sets of reals are isomorphic; in fact this conclusion is a consequence
of the Proper Forcing Axiom (PFA).
Here a linear order is \emph{$\kappa$-dense} if all of its intervals have cardinality $\kappa$.
It is not difficult to show that any suborder of $\Rbb$ of cardinality $\aleph_1$ is
bi-embeddable with an $\aleph_1$-dense set of reals and thus in Baumgartner's model,
any set of reals of cardinality $\aleph_1$ is minimal with respect to being non $\sigma$-scattered.
On the other hand, it follows easily from work of Dushnik and Miller \cite{sim_trans_LO} that
the Continuum Hypothesis (CH) implies that there are no minimal uncountable linear orders
which are separable.
(In fact Dushnik and Miller show in ZFC that there is no minimal separable linear order of cardinality
continuum.)

The main result of this paper is that Theorem \ref{Laver_thm} is consistently sharp.

\begin{thm} \label{main_thm}
If there is a supercompact cardinal,
then there is a forcing extension which satisfies CH in which there are no minimal
non $\sigma$-scattered linear orders.
\end{thm}

\noindent
This result builds on work of Moore \cite{minimal_unctbl_types}
and Ishiu-Moore \cite{no_real_Aronszajn}.
In \cite{minimal_unctbl_types} it was
proved that it is consistent with CH that
$\omega_1$ and $-\omega_1$ are the only minimal uncountable linear orderings.
In fact, this conclusion is derived from the conjunction of CH
and a certain combinatorial consequence $(\textrm{A})$ of PFA.
Notice that if $\omega_1$ and $-\omega_1$ are the only minimal
uncountable linear orders,
then any minimal non $\sigma$-scattered linear order must have the
property that it does not
contain an uncountable separable suborder or an \emph{Aronszajn suborder}.
Here an Aronszajn line is an uncountable linear order which
does not contain uncountable separable or scattered suborders.

In \cite{no_real_Aronszajn} it was proved that $\PFA^+$, a strengthening of PFA, implies
that every minimal non $\sigma$-scattered linear order is either isomorphic to a set
of reals of cardinality $\aleph_1$ or else is an \emph{Aronszajn line}.
Moreover, Martinez-Ranero \cite{A-line_wqo},
building on work of Moore \cite{linear_basis} \cite{A-line_univ} proved that
that PFA implies that the class of Aronszajn lines is well quasi-ordered by embeddability.
In \cite{no_real_Aronszajn}, it was pointed out that if the consequences of $\PFA^+$ needed
to carry out the analysis in that paper were consistent with the conjunction of $(\textrm{A})$
and CH, then
one could establish the consistency of ``there are no minimal non $\sigma$-scattered linear orders.''
In fact these consequence of $\PFA^+$ followed from a weaker axiom $\textrm{CPFA}^+$ which had
been expected to be consistent with CH; this was
later refuted in \cite{FA_CH2}.
The strategy of the present paper for proving Theorem \ref{main_thm} also utilizes
the combination of $(\textrm{A})$ and CH,
but involves a re-examination of the hypotheses sufficient to obtain the results of
\cite{no_real_Aronszajn}.

In addition to proving Theorem \ref{main_thm}, we will also establish a result concerning
the structure of non $\sigma$-scattered linear orders under the assumption of $\PFA^+$.
Baumgartner proved in \cite{new_class_otp}
that there exist non $\sigma$-scattered linear orders which do not contain
real or Aronszajn types.
His construction can be described as the lexicographic ordering on a family
$\{x_\alpha : \alpha \in S\}$ where $S \subset \omega_1$ is stationary and
$x_\alpha$ is a cofinal strictly increasing  $\omega$-sequence in $\alpha$ for each $\alpha$ in $S$.
We will refer to such a linear ordering as a \emph{Baumgartner type} and we will refer to $S$ as
its \emph{index set}.
%If $B$ is a Baumgartner type indexed by $S$ and $T \subset S$, then we will write
%$B \restriction T$ to denote the suborder with index set $T$.

\begin{thm} \label{basis_thm}
Assume $\PFA^+$ and let $X \subset \Rbb$ have cardinality $\aleph_1$ and
$C$ be a Countryman line.
If $L$ is a non $\sigma$-scattered linear order, then $L$ contains an isomorphic copy of
one of the following linear orders:
$X$, $C$, $-C$, a Baumgartner type or its reverse.
\end{thm}

The proof of Theorem \ref{main_thm} immediately yields the following result. 
\begin{thm} \label{Aronszajn_base}
It is consistent with $CH$ that no Aronszajn tree has a base of cardinality $\aleph_1$.
\end{thm}
\noindent
Here a collection $\mathcal{B}$ of uncountable downward closed subtrees of an Aronszajn tree $T$
is called a \emph{base} if whenever 
$U$ is an uncountable downward closed subtree of $T$,
there is $V \in \mathcal{B}$ such that $V \subset U$.
This answers a problem posed in \cite{bases_A-trees}, where it is proved that every Aronszajn tree
has a base of cardinality $\aleph_1$ after Levy collapsing an inaccessible cardinal to $\aleph_2$.

The paper will be organized as follows.
Section \ref{prelim:sec} will review some notation, definitions, and results
concerning linear orders.
In Section \ref{basis_thm:sec} we will prove Theorem \ref{basis_thm}.
Section \ref{main_thm:sec} contains the analysis needed to derive the conclusion of
Theorem \ref{main_thm} from a list of axioms.
Section \ref{axioms_consistent:sec} gives a proof that the collection of axioms
used in Section \ref{main_thm:sec} is consistent. This section also includes a proof of theorem \ref{Aronszajn_base}
as a remark.
The paper closes with some open problems in Section \ref{questions:sec}.

\section{Preliminaries}

\label{prelim:sec}

This section is devoted to some background and conventions on trees,
linearly ordered sets and forcing axioms.
More discussion can be found in \cite{no_real_Aronszajn}, \cite{MRP},
\cite{minimal_unctbl_types},
\cite{proper_forcing} and \cite{trees:Todorcevic}.
We will also introduce two set-theoretic axioms $(*)$ and $(\dagger)$ which will play an important
role in the proofs of Theorem \ref{main_thm} and \ref{basis_thm}.

We first recall the notion of a forcing axiom associated to a class of partial orders.
\begin{notation}
If $\mathfrak{P}$ is a class of partial orders,
then by $\FA(\mathfrak{P})$ we mean the forcing axiom for the class  
$\mathfrak{P}$: whenever $P$ is in $\mathfrak{P}$ and $\mathcal{D}$
is a collection of $\aleph_1$-many dense subsets of $P$,
there is a filter $G\subset P$ which intersects all of the dense sets in $\mathcal{D}$.
$\FA^+(\mathfrak{P})$ is the assertion that whenever $P$ is in $\mathfrak{P}$,
$\mathcal{D}$ is a collection of $\aleph_1$-many dense subsets of $P$,
and $\dot{S}$ is a name for a stationary subset of $\omega_1$,
then there is a filter $G\subset P$ which intersects
all the dense sets in $\mathcal{D}$ and satisfies that
the set
\[
\{\xi \in \omega_1 : \exists p \in G (p \Vdash \check \xi \in \dot S) \}
\]
is stationary.
\end{notation}

The following axiom is a consequence of $\FA^+(\sigma\textrm{-closed})$ and will play an important role
in our analysis of non $\sigma$-scattered linear orders of cardinality $\aleph_1$. 

\begin{defn}
$(\dagger)$ is the assertion that if 
$S \subset \omega_1$ is stationary and
for each $\alpha \in S$, $U_\alpha \subset \alpha$ is open,
then there a club $E \subset \omega_1$ such that for stationarily many $\alpha \in S \cap E$
there is an $\bar \alpha < \alpha$ such that either
$E\cap (\bar \alpha, \alpha) \subset U_\alpha$ or
$E\cap (\bar \alpha, \alpha) \cap U_\alpha = \emptyset$.
\end{defn}
Let $P$ be the poset consisting of all countable closed subsets of $\omega_1$, ordered by
end extension and let $\dot E$ be the $P$-name for the union of the generic filter.
By using the arguments of \cite{MRP}, it is possible to show that if
$S \subset \omega_1$ is stationary and $\Seq{U_\alpha : \alpha \in S}$ is as in the 
formulation of $(\dagger)$, then every condition forces
$\dot E$ satisfies the conclusion of $(\dagger)$ for $\Seq{U_\alpha : \alpha \in S}$.
In particular $\FA^+(\sigma\textrm{-closed})$ implies
$(\dagger)$.
Moreover, $(\dagger)$ holds in the model obtained by adding $\aleph_2$ Cohen subsets of $\omega_1$ to a model of GCH.
It should be noted that while this shows that it is easy to obtain models
of $(\dagger)$ and CH, it remains an open problem whether the strengthening of $(\dagger)$
in which a relative club of $\alpha \in S \cap E$ are required to satisfy the conclusion is
consistent with CH (see \cite{tutorial_NNR}).

It will often be convenient to let, for each set $X$, $\theta_X$ denote the least regular cardinal 
such that all finite iterates of the power set applied to $X$ are in $H(\theta_X)$, the collection of sets of hereditary cardinality
less than $\theta_X$.
Let $\Ecal(X)$ denote the collection of all countable elementary submodels of $H(\theta_X)$
which have $X$ as an element.

We will now recall a number of definitions from \cite{no_real_Aronszajn}.
For a linearly ordered set $L$ we will use $\hat{L}$ to denote the completion of $L$.
Formally this is the set of all Dedekind cuts of $L$ with $z$ identified with the cut 
$\{x \in L : x < z \}$. 
The purpose of the following definitions is to abstractly recover the set of indices from a
Baumgartner type, purely from its order-theoretic properties.

\begin{defn}
Whenever $L$ is a linearly ordered set and $Z$ is some arbitrary set we say that $Z$ \emph{captures} 
$x \in L$ if there is a $z \in Z \cap \hat{L}$ such that there is no element of $Z \cap L$ which is strictly in between  
$z$ and $x$. %For $x$ and $y$ in $\hat{L}$, $x \sim_Z y$ if there are only finitely many elements of $Z \cap L$ in
%between $x$ and $y$.
\end{defn}

%\begin{fact} \cite{no_real_Aronszajn}
%A linear order $L$ contains a real type if and only if there is a countable $Z$ such that
%there are is an .
%\end{fact}

\begin{fact} \label{unique} \cite{no_real_Aronszajn} 
Suppose $L$ is a linear order and let $\lambda$ be a regular cardinal such that
$\hat L$ is in $H(\lambda)$.
If $M$ is a countable elementary submodel of $H(\lambda)$ with $L \in M$ and
$x \in \hat{L} \smallsetminus M$, then
$M$ captures $x$ if and only if there is a unique $z \in \hat{L} \cap M$
such that there is no element of $M\cap L$ which is strictly in between $x$ and $z$.
In this case we say $M$ captures $x$ via $z$.
\end{fact}

\begin{defn} \cite{no_real_Aronszajn}
If $L$ is a linear order, define
$\Gamma(L)$ to be the set of all countable subsets $Z$ of $\hat{L}$ such that for some $x \in L$,
$Z$ does not capture $x$.
(This is the relative complement of the set $\Omega(L)$ in \cite{no_real_Aronszajn}.)
\end{defn}

If $B= \langle x_\alpha : \alpha \in S \rangle $ is a Baumgartner type and $M$ is a countable elementary submodel of $H(\theta)$ for some regular cardinal $\theta \geq \omega_2$, then 
$M \in \Gamma(B)$ if and only if $M\cap \omega_1 \in S$. This is because $M$ captures all elements of $B$ except $x_\delta$,
where $\delta = M\cap \omega_1$. So $\Gamma(B)$  is equivalent to $S$ modulo the equivalence induced by the following quasi-order.

\begin{defn}
Let $A,B$ be two collections of countable sets and $X=\bigcup A$,
$Y= \bigcup B$. we say $B \leq A$ 
if there is an injection $\iota : X \rightarrow Y $
such that for club many $M$ in $[Y]^{\omega}$, 
if $M \in B$ then $\iota ^{-1}M$ is in $ A$.
We let $B < A$ if $B\leq A$ but not $A\leq B$;
$A$ and $B$ are \emph{equivalent} if $A \leq B$ and $B \leq A$.
\end{defn}

The following results summarize the properties of the map $L \mapsto \Gamma(L)$ and the quasi-order $\leq$.

\begin{thm} \cite{no_real_Aronszajn} \label{sigma-scatt_char}
A linear order $L$ is not $\sigma$-scattered if and only if $\Gamma(L)$ is stationary. 
\end{thm}

\begin{prop} \cite{no_real_Aronszajn}
If $L_0$ and $L$ are linearly ordered sets and $L_0$
embeds into $L$, then $\Gamma(L_0) \leq \Gamma(L)$.
\end{prop}

A key feature of Baumgartner types $L$ is that it is always possible to find a non $\sigma$-scattered
suborder $L_0$ such that $\Gamma(L_0) < \Gamma(L)$.
This is not always possible in the more general class of non $\sigma$-scattered orders as the next
proposition shows.

\begin{prop} \label{club} \cite{no_real_Aronszajn}
If a linear order $L$ contains a real or Aronszajn type, then $\Gamma(L)$ contains a club.
\end{prop}

\begin{defn}
If $L$ is a linear order and $M$ is in $\Ecal(L)$, then
we say that an element $x$ of $L$ is \emph{internal}
(respectively \emph{external}) to $M$, if there is a club 
$E \subset [\hat{L}]^\omega$ in $M$ such that every (respectively no)
element of $E \cap M$ captures $x$.
\end{defn}

The next definition will play a central role in the proofs of our results.
It abstracts the property of Baumgartner types needed to allow us to decrease $\Gamma$ by thinning out
the linear order.

\begin{defn}
A linear order $L$ is said to be \emph{amenable}
if whenever $M$ is in $\Ecal(L)$ and $x \in L$,
then $x$ is internal to $M$.
\end{defn}
\noindent
Observe that by Theorem \ref{sigma-scatt_char}, $\sigma$-scattered linear orders are amenable.
It is also true that Baumgartner types are amenable.

\begin{prop} \cite{no_real_Aronszajn}
If $L$ is a non $\sigma$-scattered amenable linear order of cardinality $\aleph_1$ and
$S \subset \Gamma(L)$ is stationary, then there is a non $\sigma$-scattered
$L_0 \subset L$ such that $\Gamma(L_0) \leq S$.
\end{prop}

In particular, non $\sigma$-scattered amenable linear orders of cardinality $\aleph_1$ are
not minimal.
The next theorem shows that the existence of external elements of a linear order
characterizes the presence of either a real or Aronszajn suborder. In particular amenable linear orders do not contain 
real or Aronszajn types.

\begin{thm}\cite{no_real_Aronszajn} \label{external_char}
The following are equivalent for a linear order $L$:
\begin{itemize}

\item $L$ contains a real or Aronszajn type.

\item There are $M$ in $\Ecal(L)$ and $x \in L$ such that $x$ is external to $M$.

\end{itemize}
\end{thm}

We are now ready to formulate the other set-theoretic hypothesis which will be needed in our analysis.

\begin{defn}
$(*)$ is the assertion that
for every non $\sigma$-scattered linear order $L$ there is a continuous $\in$-chain 
$\langle M_\xi : \xi \in \omega_1 \rangle$
in $\Ecal(L)$ such that:
\begin{itemize}

\item
the set of all $\xi \in \omega_1$ such that $M_\xi \cap \hat{L} \in \Gamma(L)$
is a stationary set,

\item
$\hat{L_0} \subset \bigcup_{\xi \in \omega_1} M_\xi$, where 
$L_0=L\cap (\bigcup_{\xi \in \omega_1} M_\xi)$,

\item
for every $\xi$ if $M_\xi \cap \hat{L} \in \Gamma(L)$ then there is an $x \in L_0$ such that $M_\xi$ does not capture $x$.

\end{itemize}
\end{defn}

Observe that if $L_0 \subset L$ are as in the statement of $(*)$,
then $L_0$ is also non $\sigma$-scattered.
Thus $(*)$ implies every non $\sigma$-scattered linear order contains a non $\sigma$-scattered suborder
of cardinality $\aleph_1$.
Also, if we apply $(*)$ to a linear order of cardinality at most $\aleph_1$, then $L \subset \bigcup_{\xi \in \omega_1} M_\xi$
and consequently $\hat L \subset \bigcup_{\xi \in \omega_1} M_\xi$.
This gives the following fact.

\begin{fact} \label{nokurepa}
Assume $(*)$.
If $L$ is a linear order of cardinality at most $\aleph_1$ which does not contain a real type,
then $|\hat{L}| \leq \aleph_1$.
\end{fact}

In particular $(*)$ implies that CH is true.
A consequence of the work in \cite{sim_trans_LO} and \cite{minimal_unctbl_types} is that by 
iterating certain forcings over a model of CH, it is possible to obtain a generic extension
in which there is no minimal real or Aronszajn type.
We briefly review this result and recall some of the relevant definitions and terminology.
If $T$ is an Aronszajn tree, then a \emph{subtree} of $T$ is an uncountable downward closed subset of $T$. 
 
\begin{notation}
If $T$ is a tree, $t \in T$ and $\alpha$ is an ordinal, then $t\restriction \alpha$ is
defined to be $t$ if $\alpha$ is greater than the height of $t$ otherwise 
it is the unique $s\leq t$ with height $\alpha$.
\end{notation}

\begin{defn}
A sequence $\langle f_\alpha :\alpha \in \lim(\omega_1) \rangle$ is called
\emph{ladder system coloring} 
if the $\langle \dom(f_\alpha): \alpha \in \omega_1\rangle$ forms a ladder system and the range of each
$f_\alpha$ is contained in $\omega$.
\end{defn}

\begin{defn}
If $T$ is an $\omega_1$-tree, then
a ladder system coloring $\Seq{f_\alpha : \alpha \in \lim(\omega_1)}$
can be \emph{$T$-uniformized} if there is a
subtree $U$ of $T$ and  function from $\phi : U \rightarrow \omega$ such that whenever height of
$u \in U$ is a limit ordinal $\alpha$,
$f_\alpha$ agrees with $\xi \mapsto \phi (u\restriction \xi)$ at all except 
for finitely many $\xi \in \dom(f_\alpha)$.
\end{defn}
\begin{defn}
$(\textrm{A})$ is the assertion that every ladder system coloring can be 
$T$-uniformized for every Aronszajn tree $T$.
\end{defn}

The significance of $(\textrm{A})$ lies in the following theorem, along with the fact that it is
consistent with CH.

\begin{thm}\cite{minimal_unctbl_types} \label{Aronszajn}
Assume $(\textrm{A})$ and $2^{\aleph_0} < 2^{\aleph_1}$.
There is no minimal Aronszajn line.
\end{thm}

In \cite{minimal_unctbl_types}, a forcing $Q_{T,\bar{f}}$ was introduced which $T$-uniformizes
a given ladder system coloring $\bar f$.
We will recall the definition of this poset in Section \ref{axioms_consistent:sec}
when we need to analyze it, but
for now we will simply summarize its important properties.

While $(<\!\omega_1)$-properness and complete properness play a role
in the proof of the main result of this paper,
they can be treated as black boxes via the following results, along
with the straightforward fact that $\sigma$-closed posets are
both $(<\!\omega_1)$-proper and completely proper.

\begin{lem}\cite{minimal_unctbl_types} \label{Qtf_lem}
For every ladder system coloring $\bar{f}$ and Aronszajn tree $T$,
the forcing $Q_{T,\bar{f}}$ is completely proper and $(<\!\omega_1)$-proper. 
\end{lem}

\begin{thm} \cite{proper_forcing} \label{nnr}
A countable support iteration of $(<\!\omega_1)$-proper, completely proper forcing is proper
and does not introduce new real numbers.
\end{thm}

We will also need the following iteration theorem of Shelah.

\begin{thm}\label{nnb}\cite[III.8.5]{proper_forcing} 
If the iterands of a countable support iteration are proper and don't add new uncountable branches to 
$\omega_1$-trees, then the iteration is proper and does not add uncountable branches to $\omega_1$-trees. 
\end{thm}

\section{A rough classification of non $\sigma$-scattered orders}
\label{basis_thm:sec}

In \cite{no_real_Aronszajn} it was shown that under $\PFA^+$, every non $\sigma$-scattered
linear order contains an amenable non $\sigma$-scattered suborder of cardinality $\aleph_1$.
In this section we prove that under a fragment of $\PFA^+$
every non $\sigma$-scattered amenable linear order contains
a copy of a Baumgartner type or its reverse.
Taken together, these results
determines a basis for the class of non $\sigma$-scattered linear orders under $\PFA^+$:
if $X$ is any set of reals of cardinality $\aleph_1$
and $C$ is any Countryman type, then any non $\sigma$-scattered linear order must
contain an isomorphic copy of either $X$, $C$, $-C$, or a Baumgartner type of cardinality $\aleph_1$ or its reverse.

\begin{thm} \label{basis}
Assume the conjunction of $\MA_{\omega_1}$ and  $(\dagger)$.
If $L$ is an amenable non $\sigma$-scattered linear order of size $\aleph_1$, then it contains a copy of a Baumgartner type
or its reverse.
\end{thm}

First we will prove the following lemma.

\begin{lem} \label{same_capture}
Suppose that $L$ is a an amenable linear order of cardinality $\aleph_1$.
If $\langle M_\xi : \xi \in \omega_1 \rangle$
is a continuous $\in$-chain of elements of $\Ecal(L)$
which is in $N \in \Ecal(L)$ and $N\cap \omega_1 = \delta$,
then $M_\delta$ and $N$ capture the same elements of $L$.
\end{lem}

\begin{proof}
First observe that by continuity of the $\in$-chain and the fact that 
$\{\nu \in \omega_1 : M_\nu \cap \omega_1 = \nu\}$ is a club in $N$,
$M_\delta \subset N$ and $M_\delta \cap \omega_1 = \delta$.
Next observe that since $L$ has cardinality $\aleph_1$, $N \cap L = M_\delta \cap L$ and thus
any element of $L$ captured by $M_\delta$ is captured by $N$.
Now suppose that $N$ captures $x \in L$ and let $z \in \hat{L} \cap N$ be such that 
there is no element of $N\cap L$ which is strictly in between $x$ and $z$.
Since $L$ is amenable, there is a club $E \subset [\hat{L}]^{\omega}$ in $M_\delta$
such that for all $Z\in M_\delta \cap E$, $Z$ captures $x$.
Let $\lambda \in \theta_L \cap M_0$ be a regular cardinal such that the powerset of $[\hat{L}]^\omega$ is in
$H(\lambda)$.
Let $\overline{M} \in N$ be a countable elementary submodel of $H(\lambda)$ such that
$\langle M_\xi \cap H(\lambda) : \xi \in \omega_1 \rangle$, $E$, and $z$ are in $\overline{M}$.
Observe that for sufficiently large $\xi < \delta$,
$M_\xi \cap \hat{L}$ is in $E$ and if 
$\nu = \overline{M} \cap \omega_1$ then $L \cap \overline{M} = M_\nu \cap L$.
Notice that $\overline{M}$ captures $x$ via $z$.
Since $M_\nu \cap \hat L$ is in $E \cap M_\delta$, $M_\nu$ also captures $x$.
By Fact \ref{unique}, it must be that $z$ is in $M_\nu$ and hence $M_\delta$. 
\end{proof}

\begin{proof}[Proof of Theorem \ref{basis}]
Now let $\langle M_\xi : \xi \in \omega_1 \rangle$ be a continuous $\in$-chain 
of elements of $\Ecal(L)$.
Since $L$ is amenable it does not contain any
real types, there is a countable set $X_\xi \subset L$ such that
if $M_\xi \cap L \subset X_\xi$ and if $y \in L \setminus M_\xi$, there is a unique
$x \in X_\xi \setminus M_\xi$ such that $x \ne y$.
Let $x:\omega \times \omega_1 \rightarrow L$ be such that 
for all $\xi \in \omega_1$, $X_\xi = \{x(n,M_\xi \cap \omega_1): n\in \omega \}$.
Now let $\langle N_\xi : \xi \in \omega_1 \rangle$
be a continuous $\in$-chain of elements of $\Ecal(L)$ such that
$\langle M_\xi : \xi \in \omega_1 \rangle$ and $x$ are in $N_0$. 
Note that there is a club of $\xi$ in $\omega_1$ such that 
$M_\xi \cap \omega_1 =\xi=N_\xi \cap \omega_1$ and
hence $M_\xi$ and $N_\xi$ capture the same elements of $L$.
Since $\Gamma(L)$ is stationary, then by applying the pressing down lemma
there is a stationary set $S_0 \subset \omega_1$, an $n \in \omega$, and a club
$E\subset [\hat{L}]^{\omega}$ such that if $\xi \in S_0$:
\begin{itemize}

\item $M_\xi \cap \omega_1 =\xi=N_\xi \cap \omega_1$;

\item $x(n,\xi)$ is not captured by $N_\xi$;

\item $E$ is in $N_\xi$ and if $Z$ is in $N_\xi \cap E$, then $Z$ captures
$x(n,\xi)$.

\end{itemize}
Set $y_\xi =x(n,\xi)$ for all $\xi \in S_0$.
Now it is easy to see that for all $\xi$ and $\eta$ in $S_0$, $N_\xi$ captures
$y_\eta$ if and only if $\xi \neq \eta$.

Let $z_\xi$ $(\xi \in \omega_1)$ be an enumeration of all
$z \in \hat{L}$ for which there are
$\eta \in \omega_1$ and $\alpha \in S_0$ such that $N_\eta$ captures $y_\alpha$ via $z$.
We can assume without loss of generality that this enumeration is in $N_0$.
For every $\alpha \in S_0$ define $g_\alpha :\alpha \rightarrow \{z_\xi: \xi \in \omega_1 \}$
by letting $g_\alpha(\xi)$ be the unique $z \in N_\xi$
such that $N_\xi$ captures $y_\alpha$ via $z$.
Note that if $g_{\alpha}(\xi)= z_\eta$ then $\eta \in \alpha$.

\begin{claim} \label{g_claim}
The following are true for $\alpha,\beta \in S_0$:
\begin{enumerate}

\item \label{otp_omega}
$\{ \xi \in \alpha : g_\alpha (\xi) \neq g_\alpha (\xi +1)\}$
has order type $\omega$ and supremum $\alpha$

\item
If $y_\alpha < y_\beta$, then $g_\alpha(\xi) \leq g_\beta(\xi)$ for all $\xi < \min (\alpha,\beta)$.

\item
If $\alpha < \beta$, then there is a $\xi < \alpha$ such that
$g_\alpha(\xi) \ne g_\beta(\xi)$.

\item
If $\xi < \eta < \min(\alpha,\beta)$ and
$g_\alpha(\xi) \ne g_\beta(\xi)$, then $g_\alpha(\eta) \ne g_\beta(\eta)$.

\end{enumerate}
\end{claim}

\begin{proof}
First observe that for each $\alpha \in S_0$ and limit $\eta \in \alpha$
there is an $\bar \eta \in \eta$
such that $g_\alpha \restriction (\bar \eta, \eta]$ is constant.
On the other hand $N_\alpha$ does not capture $y_\alpha$ and therefore the set
$\{ \xi \in \alpha : g_\alpha (\xi) \neq g_\alpha (\xi +1)\}$
must have ordertype $\omega$ and supremum $\alpha$.
This proves (\ref{otp_omega}); the remainder of the items follow easily from (\ref{otp_omega}) and
the definitions.
\end{proof}

Define $C_\alpha$ to be the set of all $\xi \in \alpha$ such that $z_\xi$ is in the range of
$g_\alpha$ and equip the set 
$\{C_\alpha : \alpha \in S_0 \}$ with the lexicographic order.
For each $\alpha \in S_0$ let $U_\alpha = \{\xi \in \alpha: g_\alpha(\xi) < y_\alpha \}$ and
observe that $U_\alpha$ is an open subset of $\alpha$.
So by $(\dagger)$ there is a stationary set $S \subset S_0$ such that either
\begin{itemize}

\item  for every $\alpha \in S$ and $\xi \in S \cap \alpha$, $g_\alpha (\xi) > y_\alpha$ or,

\item  for every $\alpha \in S$ and $\xi \in S \cap \alpha$, $g_\alpha (\xi) < y_\alpha$.

\end{itemize}
Without loss of generality assume that for every $\alpha \in S$ and $ \xi \in \alpha \cap S$, 
$g_\alpha (\xi) > y_\alpha$.
Define $S'$ to be the set of all elements of $S$ which are limit points of elements of $S$.
 
Let $Q$ be the set of all finite $p\subset S'$ such that whenever $\alpha \neq \beta$ are in $p$,
$C_\alpha <_{\lex} C_\beta$ if and only if $y_\alpha < y_\beta $.
We will prove that $Q$ is c.c.c..
 
Suppose for a contradiction that $X$ is an uncountable antichain in $Q$.
By applying the $\Delta$-System Lemma and removing the root if necessary, we may assume that
$X$ is pairwise disjoint and consists of elements of some fixed cardinality $n$.
Let $M$ be an element of $\Ecal(Q)$ such that $X$, $L$, $x$,
$\Seq{N_\xi : \xi \in \omega_1}$, $\Seq{y_\xi : \xi \in S}$, and
$\Seq{z_\xi : \xi \in \omega_1}$ are all in $M$.
Set $\delta = M \cap \omega_1$ and let
$p=\{\alpha_1, ..., \alpha_n\}$ be in $X$ such that
$\delta < \alpha_i$ for all $i \leq n$.
Let $\zeta \in \delta \cap S$ be such that:
\begin{itemize}

\item if $i,j \leq n$, then $g_{\alpha_i} \restriction \delta \ne g_{\alpha_j} \restriction \delta$
implies $g_{\alpha_i} (\zeta) \ne g_{\alpha_j} (\zeta)$;

\item the range of $g_{\alpha_i} \restriction \zeta+1$ coincides with the range of 
$g_{\alpha_i} \restriction \delta$ for each $i \leq n$
(i.e. $C_{\alpha_i} \cap \delta \subset \zeta+1$ for each $i \leq n$).

\end{itemize}
Notice that the existence of $\zeta$ follows from the observation that if $g_\alpha(\xi) \ne g_\beta(\xi)$,
then $g_\alpha(\eta) \ne g_\beta(\eta)$ for all $\eta > \xi$.
By elementarity of $M$ there exists a $p'=\{\alpha'_1,...,\alpha'_n\}$ in $M\cap X$ such that:
\begin{itemize}

\item for all $i,j \leq n$, $y_{\alpha_i} < y_{\alpha_j}$ if and only if $y_{\alpha_i'} <y_{\alpha_j'}$;

\item
if $i \leq n$, then
$g_{\alpha_i} (\zeta) = g_{\alpha_i'}(\zeta)$.

\end{itemize}
We will now show that $p \cup p' \in Q$.
Let $i,j \leq n$.
There are two cases, depending on whether $g_{\alpha_i} (\zeta)$ and $g_{\alpha_j}(\zeta)$ are the same.
If $g_{\alpha_i} (\zeta) \ne g_{\alpha_j}(\zeta)$, then observe that
$g_{\alpha_j}(\zeta) = g_{\alpha_j'}(\zeta)$ and
\[
C_{\alpha_i} \cap (\zeta+1) \ne C_{\alpha_j} \cap (\zeta+1) = C_{\alpha_j'} \cap (\zeta+1).
\]
Since $p$ and $p'$ are both in $Q$, it follows that
$y_{\alpha_i} < y_{\alpha_j'}$ is equivalent to $C_{\alpha_i} <_{\lex} C_{\alpha_j'}$

If $g_{\alpha_i}(\zeta) = g_{\alpha_j}(\zeta)$, then observe that
$g_{\alpha_i} \restriction \delta = g_{\alpha_j} \restriction \delta$ and thus that
$C_{\alpha_i} \cap \delta = C_{\alpha_j} \cap \delta$.
Observe that
\[
g_{\alpha_j'} \restriction \zeta = g_{\alpha_j} \restriction \zeta = g_{\alpha_i} \restriction \zeta
\]
and that $g_{\alpha_i}$ is constant on the interval $[\zeta,\delta)$.
Also, $g_{\alpha_j'}$ is not constant on $[\zeta,\delta)$ by Claim \ref{g_claim}.
Observe that there is a $\xi \in S$ such that $\zeta < \xi < \alpha_j'$ and 
\[
g_{\alpha_j'}(\xi) < g_{\alpha_j'}(\zeta) = g_{\alpha_j}(\zeta) = g_{\alpha_j}(\xi) = g_{\alpha_i}(\xi)
\]
It follows that $y_{\alpha_i} > y_{\alpha_j'}$.
On the other hand,
\[
 C_{\alpha_i} \cap \delta = C_{\alpha_i} \cap (\zeta +1)  = C_{\alpha_j'}\cap (\zeta +1) \ne C_{\alpha_j'} \cap \delta
\]
and consequently $C_{\alpha_j'} <_{\lex} C_{\alpha_i}$. 
Since $i,j \leq n$ were arbitrary, $p$ and $p'$ are compatible and thus $Q$ is c.c.c..

By applying $\MA_{\aleph_1}$ to the finite support product $Q^{< \omega}$ of countably many copies of
$Q$, it is possible to find a partition of 
$S$ into countably many pieces such that whenever $\alpha$ and $\beta$
are in the same piece of the partition,
$C_\alpha <_{\lex} C_\beta$ if and only if $y_\alpha < y_\beta$.
Since there is a piece of this partition which is stationary,
it shows that $L$ contains a Baumgartner type.
\end{proof}

We finish this section by noting if we add a Cohen real $r$ to a model of ZFC,
then Theorem \ref{basis} will not hold in the resulting generic
extension.
To see this, suppose that $r \in 2^\omega$ and $\Seq{x_\xi : \xi \in \lim(\omega_1)}$ is such that
$x_\xi:\omega \to \xi$ is increasing and has cofinal range for each $\xi$.
Define a linear ordering on $\lim(\omega_1)$ by $\xi <_r \eta$ if and only if
$$
x_\xi (n) < x_\eta (n) \textrm{ is equivalent to } r(n) = 0
$$
where $n$ is minimal such that $x_\xi(n) \ne x_\eta(n)$.
It is left to the reader to check that if $S \subset \lim(\omega_1)$ is stationary,
then there is a comeager set of $r$ such that
$(S,<_r)$ contains both a copy of $\omega_1$ and of $-\omega_1$.
Furthermore, if $S$ is non-stationary, then $(S,<_r)$ is $\sigma$-scattered and thus not a Baumgartner
type.
On the other hand, it is not hard to show that every uncountable subset of a Baugmartner type
contains a copy of $\omega_1$;
in particular, Baumgartner types do not contain $-\omega_1$.  
Since every stationary subset of $\omega_1$ in the generic extension by a Cohen real contains a ground model
stationary set, this proves the claim.

\section{An axiomatic analysis of non $\sigma$-scattered orders}

\label{main_thm:sec}
In this section we will prove the following proposition.

\begin{prop}
Assume $(\dagger)$ and $(*)$.
If $L$ is a non $\sigma$-scattered linear order
which does not contain a real or Aronszajn type, then 
there is a non $\sigma$-scattered suborder $L' \subset L$ with $\Gamma(L') < \Gamma(L)$.
\end{prop}

\begin{proof}
As noted in Section \ref{prelim:sec},
$(*)$ implies that $L$ contains a non $\sigma$-scattered suborder $L_0$ such that
$\hat L_0$ has cardinality $\aleph_1$.
We may therefore assume without loss of generality that $|L| = |\hat L| = \aleph_1$.
This implies, in particular that if $M$ and $N$ are in $\Ecal(L)$ and 
$M\cap \omega_1=N\cap \omega_1$, then $M\cap \hat{L}=N\cap \hat{L}$.
If $Z \subset \hat{L}$ is countable, let $\{x(n,Z):n \in \omega\} \subset L$ be such that
$Z \cap L \subset \{x(n,Z):n \in \omega\}$ and if $y \in L \setminus Z$, then there is an 
$n$ such that no element of $L \cap Z$ is between $x(n,Z)$ and $y$.
This is possible since $L$ does not contain a real type.
Let $\langle N_\xi : \xi \in \omega_1 \rangle$ be a continuous $\in$-chain in 
$\Ecal(L)$ with the map $Z \mapsto \{x(n,Z): n \in \omega\}$ in $N_0$.
Since $L$ is not $\sigma$-scattered, there is an $n\in \omega$ such that
\[
S_0=\{\xi \in \omega_1 : N_\xi \cap \omega_1=\xi \textrm{ and } N_\xi \textrm{ does not capture } x(n,N_\xi \cap \hat{L}) \}
\] 
is stationary.
Fix such an $n$ and set $x_\xi=  x(n,N_\xi \cap \hat{L})$.
For each $\alpha \in S_0$ let $U_\alpha$ be the set of all $\xi \in \alpha$ 
such that $N_\xi$ captures $x_\alpha$.
Clearly $U_\alpha$
is an open subset of $\alpha$ so by $(\dagger)$ there is a stationary subset $S\subset S_0$ and a club
$E\subset \omega_1$ such that for every $\alpha \in S$ there is an $\bar \alpha \in \alpha$
such that either
$E \cap (\bar \alpha, \alpha) \subset U_\alpha$ or
$E \cap (\bar \alpha, \alpha) \cap U_\alpha =\emptyset$.
The second alternative
can only happen for at most nonstationary many $\alpha \in S$, because $L$ has no external element
by Theorem \ref{external_char}.
By applying the Pressing Down Lemma
and thinning $S$ down if necessary, we can assume that for every $\alpha , \beta \in S$, 
$N_\alpha$ captures $x_\beta$ if and only if $\alpha \neq \beta$.

Now let $S'\subset S$ be stationary such that $S\smallsetminus S'$ is also stationary
and define $L' = \{x_\xi : \xi \in S'\}$.
We will show that $L'$ is not $\sigma$-scattered and that $\Gamma(L')< \Gamma(L)$. 

Fix an $M\in \Ecal(L)$ which has $\langle N_\xi : \xi \in \omega_1 \rangle$, $S$, and $S'$ as elements
and has the property that $M\cap \omega_1 = \delta \in S'$.
To show that $L'$ is not $\sigma$-scattered we prove that $M$ does not capture $x_\delta$ in $L'$.
Suppose for contradiction that $M$ captures $x_\delta$ in $L'$ via $z \in \hat L \cap M$.
By replacing $L$ with $-L$ if necessary, we may assume that $z < x_\delta$.
Let 
\[
A= \{x_\xi: \xi \in S' \textrm{ and }z < x_\xi \}.
\]
Observe that $A$ is in $M$ and hence $\inf(A)$ is also in $M$.
Since $\inf(A) \leq x_\delta$ and $x_\delta$ is not in $M$,
it follows that $\inf(A) < x_\delta$.
Since $M$ does not capture $x_\delta$ in $L$, 
there is $y \in L \cap M$ such that $z \leq \inf (A) < y < x_\delta$.
By elementarity of $M$, there is a $\xi \in S' \cap M$ such that $z < x_\xi < y < x_\delta$.
But this contradicts our assumption that $M$ captures $x_\delta$ in $L'$ via $z$.
 
To see that $\Gamma(L) \nleq \Gamma(L')$ it suffices to show that the set of all $M \in \Ecal(L)$ which capture
all elements of $L'$ but does not capture some elements of $L$ forms a stationary set.
To this end let
$M \in \Ecal(L)$ with $L' \in M$ and $M\cap \omega_1 \in S\smallsetminus S'$
and observe that $M$ does not capture 
$x_{\delta}$ in $L$ but it captures all elements of $L'$. 
\end{proof}

\section{The consistency of the axioms}
\label{axioms_consistent:sec}

In this section we will prove that if there is a supercompact cardinal, then there is a forcing
extension with the same reals which satisfies $(*)$, $(\dagger)$, and $(\textrm{A})$.
By results of the previous section this will finish the proof of Theorem \ref{main_thm}.
Our forcing construction will resemble the consistency proof of $\PFA^+$ and will
involve a countable support iteration of forcings which are
completely proper, $(<\!\omega_1)$-proper, and which do not add new uncountable branches
through $\omega_1$-trees.
By results of Shelah discussed in the introduction, the resulting iteration will not introduce new
reals or uncountable branches through $\omega_1$-trees.

All of the iterands used in building the iteration will either be $\sigma$-closed or else
be of the following form.

\begin{defn}\cite{minimal_unctbl_types}
For an Aronszajn tree $T$ and ladder system coloring $\bar{f}$ let $Q_{T, \bar{f}}$
be the set of all conditions $q=(\phi_q , \mathcal{U}_q)$ such that:
\begin{itemize}

\item $\phi_q$ is a function from $X_q \subset T$ into $\omega$ such that
$X_q$ is a countable downward closed subset of $T$
which has a last level of height $\alpha_q$,

\item if $t \in X_q$ has limit height  $\delta$, 
$f_\delta$ agrees with  $\xi \mapsto \phi_q (t\restriction \xi)$ at all $\xi \in C_\alpha$ except 
for finitely many $\xi \in C_\alpha$.

\item $\mathcal{U}_q$ is a non-empty countable collection of pruned subtrees of $T^{[n]}$ for some $n$.

\item for every $U \in \mathcal{U}_q$ there is some $\sigma \in U$ which is a subset of the last 
level of $X_q$.

\end{itemize}
($T^{[n]}$ is the collection of all weakly increasing $n$-tuples from some level of $T$,
regarded as a tree with the coordinatewise order.)
We let $p \leq q $, in $Q$ if $X_p \restriction \alpha_q = X_q$, $\mathcal{U}_q \subset \mathcal{U}_p$,
 and $\phi_p \restriction X_q =\phi_q$.
\end{defn}

\begin{rmk}
A simplification of this type of forcings can be used to prove Theorem \ref{Aronszajn_base}. 
For an Aronszajn tree $T$  let $Q_T$ be the set of all conditions as above forgetting the information about the ladder system 
coloring. More precisely $Q_T$ consists of all conditions $q=(X_q , \mathcal{U}_q)$ such that,
\begin{itemize}
\item
$X_q$ is a countable downward closed subset of $T$
which has a last level of height $\alpha_q$,
\item
$\mathcal{U}_q$ is a non-empty countable collection of pruned subtrees of $T^{[n]}$ for some $n$.
\item 
for every $U \in \mathcal{U}_q$ there is some $\sigma \in U$ which is a subset of the last 
level of $X_q$.
\end{itemize}
We let $p \leq q $, in $Q$ if $X_p \restriction \alpha_q = X_q$, $\mathcal{U}_q \subset \mathcal{U}_p$. It is easy to see 
that the forcing  $Q_{T, \bar{f}}$ projects onto $Q_T$ for every Aronszajn tree $T$, so by the work in \cite{minimal_unctbl_types},
$Q_T$ is completely proper, $<\omega_1$-proper and satisfies proper isomorphism condition. Now let $\mathbb{P}$ be the countable support iteration of all posets of $Q_T$ of length $\omega_2$ such that whenever $T$ is an Aronszajn tree in some 
intermediate model, $Q_T$ is repeated in the iteration cofinally often. Let $\mathbb{V}$ be a model satisfying 
$2^\omega = \omega_1 + 2^{\omega_1}=\omega_2$, and let $G$ be $\mathbb{P}$-generic over $\mathbb{V}$. Then it is easy to see that $\omega_2$ is preserved and in $\mathbb{V}[G]$
\begin{itemize}
\item $2^\omega = \omega_1 + 2^{\omega_1}=\omega_2$,
\item if $T$ is an Aronszajn tree, there is a sequence $\langle V_i : i \in \omega_2 \rangle$ of uncountable downward closed 
subtrees of $T$ such that whenever $i \in j$, $V_i$ contains no subtree of $V_j$.   

\end{itemize}
This proves Theorem \ref{Aronszajn_base}.

\end{rmk}
The following lemma asserts that these forcings $Q_{T,\bar{f}}$ do not add new uncountable branches to 
$\omega_1$-trees.

\begin{lem} \label{Qtf_good}
Suppose $T$ is Aronszajn and $S$ is an $\omega_1$-tree, and 
$\bar{f}$ is a ladder system coloring.
Then $Q_{T,\bar{f}}$ does not add new uncountable branches to $S$.
Consequently, if $L$ is a linear order of size $\aleph_1$, then forcing with $Q_{T,\bar f}$ does
not introduce new elements to $\hat L$.
\end{lem}

\begin{proof}
Let $Q$ denote $Q_{T,\bar f}$ and let $\dot{b}$ be a $Q$-name which is forced by some $p \in Q$ to
be an uncountable branch in $S$ which is not in the ground model.
If $q$ is in $Q$ and $\sigma$ is in $T^{[n]}$ for some $n$, then we say that
$\sigma$ is \emph{consistent} with $q$ if the range of $\sigma \restriction \alpha_q$
is contained in $X_q$.

Let $M \in \Ecal(Q)$ with $p, \dot{b} \in M$ and set
$\delta = M \cap \omega_1$. 

\begin{claim}
If $\sigma \in T_\delta^{[n]}$ is consistent with $p$
and $s \in S_\delta$, then there is a condition 
$q \leq p$ in $M \cap Q$ such that 
$q \Vdash \check{s} \notin \dot{b}$ and 
such that $\sigma$ is consistent with $q$.
\end{claim}

\begin{proof}
By Lemma 5.5 in \cite{minimal_unctbl_types} we can find a decreasing sequence 
$\langle p_k: k \in \omega \rangle$ in $M$ such that:
\begin{itemize}

\item $p_0 =p$,

\item $p_{k+1}$ decides $\dot{b} \restriction \alpha_{p_k}$,

\item $\sigma$ is consistent with $p_k$ for all $k$,

\item $\langle p_k: k \in \omega \rangle$ has a lower bound in $M$.

\end{itemize}
\noindent
Thus without loss of generality we can assume that 
$p$ forces ${\check s \restriction \check \alpha_p} \in \dot{b}$.

Suppose for contradiction that for every $q \leq p$ in $M \cap Q$,
if $q \Vdash \check s \not \in \dot b$, then $\sigma$ is not consistent with
$q$.
Define $W$ to be the set of all $\tau \in T^{[n]}$ which are compatible with $\sigma \restriction \alpha_p$ 
and such that there exists an $\bar s \in S_{\Ht(\tau)}$ compatible with
$s \restriction \alpha_p$ and for all $q \leq p$,
if $q \Vdash \bar s \not \in \dot b$ and $\alpha_q \leq \Ht(\tau)$, then
$\range (\tau \restriction \alpha_q) \not \subset X_q$.
Since $W$ is definable from parameters in $M$, it is in $M$.
Observe that $W$ is downwards closed and that 
$s$ witnesses that $\sigma$ is in $W$.
Hence
by elementarily of $M$, $W$ is uncountable.
%Define $U$ to be $W$ together with all predecessors of $\sigma \restriction \alpha_p$.
Let $U$ be the set of all $\tau \in W$ which have uncountably many extensions in $W$.
Notice that $\sigma \restriction \alpha_p$ is in $U$ and thus
$p'=(\varphi_p,X_p,\mathcal{U}_p \cup \{U\})$ is a condition in $Q$.
 
For each $\tau \in U_\delta$ and $t \in S_\delta$, let $\varphi (\tau,t)$ be the assertion:
\emph{``whenever $r \leq p'$  is
$(M,Q)$-generic with 
$\range (\tau) \subset X_r$, 
$r \Vdash \check{t}\in \dot{b}$.''}
Notice that if $r$ is $(M,Q)$-generic, then so is $r \restriction \delta$.
It is easy to see that for every 
$\tau \in W_\delta$ there exists a unique $t \in S_\delta$ which extends $s\restriction \alpha_p$
such that 
$\varphi(\tau,t)$.
Moreover, observe that if $\tau_1$ and $\tau_2$ are in $U_\delta$ and $s_1,s_2$ are such that
$\phi(\tau_1,s_1)$ and $\phi(\tau_2,s_2)$, then we can find an $(M,Q)$-generic condition $r\leq p'$
which is consistent to both $\tau_1$ and $\tau_2$.
This implies $r\Vdash \check{s_1}= \check{s_2}$.
Thus $s \in S_\delta$ satisfies that $\phi(\tau,s)$ holds for every
$\tau \in U_\delta$.

%Therefore we can fix an $s$ in
%$S_\delta$ which extends  $s\restriction \alpha_p$ such that for all 
%$\tau \in W_\delta$, $\varphi (\tau, s)$ holds.

We now claim that $p' \Vdash \check{s}' \in \dot{b}$ for all $s' < s$.
Since such an $s'$ is necessarily in $M$, by elementarity it suffices to show
that if $p'' \leq p'$ is in $M$, then $p''$ has an extension $r$ which forces that
$\check{s}' \in \dot{b}$.
Because $Q$ is proper, $p''$ has an $(M,Q)$-generic extension $r$.
Let $\tau \in U_\delta$ be such that $\range (\tau) \subseteq X_r$.
Since $r\leq p'$ and $\varphi (\tau, s')$ holds, $r\Vdash \check{s}' \in \dot{b}$.
Thus we have established that $p' \Vdash \check{s}' \in \dot{b}$ for all
$s' < s$.
By elementarily, $\{ t \in S : p' \Vdash \check{t} \in \dot b\}$ is uncountable,
which implies that $p'$ decides $\dot b$, a contradiction.
\end{proof}

Returning to the main proof, by the claim we can find a condition
$\bar{p} \leq p$ such that $\bar{p}\Vdash S_\delta \cap \dot{b} = \emptyset$.
\end{proof}

\begin{thm}
Assume there is a supercompact cardinal.
Then there is forcing extension in which $(\textrm{A})$, $(\dagger)$, and $(*)$ hold.
\end{thm}

\begin{proof}
Let $V$ be a ground model with a supercompact cardinal $\kappa$.
By performing some preparatory forcing if necessary, we may assume that CH is true.
Mimicking the consistency proof of PFA (see \cite{yorkshireman} or \cite{set_theory:Jech}),
use a Laver function $\psi$ to build a countable support iteration
$\Seq{P_\alpha, \dot Q_\alpha : \alpha \in \kappa}$ such that:
\begin{itemize}

\item $\dot Q_\alpha$ is a $P_\alpha$-name in $V_\kappa$
for a partial order which is either $\sigma$-closed
or of the form $Q_{\dot T,\bar f}$;

\item if $\psi(\alpha)$ is a $P_\alpha$-name and $p \in P_\alpha$ forces that $\psi(\alpha)$ either $\sigma$-closed or of the form $Q_{\dot T,\bar f}$, then $p$ forces
$\dot Q_\alpha = \psi(\alpha)$.

\end{itemize}
Let $G \subset P_\kappa$ be a $V$-generic filter.
It is immediate that $V[G]$ satisfies $(\textrm{A})$.
By Lemma \ref{Qtf_lem} and Theorem \ref{nnr} the iteration does not add new reals and
thus $V[G]$ satisfies CH.
By Lemmas \ref{nnb} and \ref{Qtf_good}, every final segment of the iteration does not add new uncountable branches
to $\omega_1$-trees.
Arguing as in \cite{yorkshireman}, $V[G]$ satisfies $\FA^+(\sigma\textrm{-closed})$
and in particular $(\dagger)$.
  
We will now show that $(*)$ holds in $V[G]$.
Fix for a moment a non $\sigma$-scattered linear order $L$ in $V[G]$ and
let $Q$ be the set of all countable continuous $\in$-chains in $\Ecal(L)$ ordered by end extension. 
It is obvious that $Q$ is $\sigma$-closed and easily verified that
\[
\dot{S}=\{(\check{\xi},q): \xi \in \dom(q) \textrm{ and } q(\xi) \cap \hat L \in \Gamma(L)\}
\]
is a $Q$-name for a stationary subset of $\omega_1$.
Since $Q$ is countably closed, it does not add new elements to $\hat L$.
Thus if $H \subset Q$ is a $V[G]$-generic filter,
then $V[G][H]$ contains the desired witness to $(*)$ for $L$.
Moreover, this witness is preserved in any further generic extension by a proper forcing
in which $\hat L$ does not gain new elements.

The proof that $\FA^+(\sigma\textrm{-closed})$ holds in $V[G]$ can now be applied
in this situation to show that $(*)$ holds in $V$.
The only difference is that while in the verification of $\FA^+(\sigma\textrm{-closed})$
it is sufficient to know that the factor forcings are proper,
in our setting it is necessary to know that, additionally, the factor forcings
do not add new elements to the completions of linear orders.
As noted already, this follows from Lemmas \ref{nnb} and \ref{Qtf_good}.
\end{proof}

\section{Open questions}

\label{questions:sec}

We will conclude this paper  by mentioning some open questions which are natural in light of the
results obtained here and in \cite{no_real_Aronszajn}.
The first is a problem of Galvin.

\begin{question}(see \cite[Problem 4]{new_class_otp})
Must every minimal non $\sigma$-scattered linear order be a real type nor an Aronszajn type?
\end{question}

Of course it is consistent that this question has positive answer (this was first shown in \cite{no_real_Aronszajn}),
but at present it seems possible that this question could have a positive answer in ZFC (we conjecture, however,
that a negative answer is consistent).
We also don't know the answer to the following question.

\begin{question}
Must every minimal non $\sigma$-scattered linear order have cardinality $\aleph_1$?
\end{question}

Notice that if $\kappa > \aleph_1$ is a regular cardinal 
and $L = \{x_\alpha : \alpha \in S\}$
is a ladder system indexed by a nonreflecting stationary set $S \subset \kappa$.
consisting of ordinals of countable cofinality,
then the lexicographical ordering on $L$ is non $\sigma$-scattered but has no $\sigma$-scattered suborder
of cardinality $\aleph_1$ (of course this example fails to be minimal).

Finally, it is unclear whether Theorem \ref{basis_thm} can be sharpened so that the Baumgartner types
are all realized as suborders of a single Baumgartner type.

\begin{question}
Assume $\PFA^+$.
If two Baumgartner types are indexed by a common stationary subset of $\omega_1$,
must there be a non $\sigma$-scattered order which embeds into both of them?
\end{question}

\section*{Acknowledgements}
The research presented in this paper was supported in part by NSF grants DMS-1262019 and DMS-1600635.

%\bibliography{../global}{}
%\bibliographystyle{mrl}

\def\Dbar{\leavevmode\lower.6ex\hbox to 0pt{\hskip-.23ex \accent"16\hss}D}

\end{document}